\documentclass[11pt,letterpaper]{article}

\usepackage[margin=1in]{geometry}
\usepackage{amsmath,amssymb,amsthm}
\usepackage{booktabs}
\usepackage{float}
\usepackage{placeins}
\usepackage{needspace}
\usepackage{tikz}
\usepackage{xcolor}
\usepackage{listings}
\usepackage{microtype}
\usepackage[hidelinks,
            pdftitle={The smallest square tileable by pairwise incomparable integer
                      rectangles},
            pdfsubject={Discrete geometry; Problem C5 of Croft, Falconer and Guy,
                        Unsolved Problems in Geometry},
            pdfauthor={George M. Georgiou},
            pdfkeywords={incomparable rectangles, tiling, dissection, exhaustive search,
                         computer-assisted proof}]{hyperref}

\theoremstyle{plain}
\newtheorem{theorem}{Theorem}
\newtheorem{lemma}{Lemma}
\newtheorem{proposition}{Proposition}
\newtheorem{corollary}{Corollary}
\newtheorem{obs}{Observation}
\theoremstyle{remark}
\newtheorem*{remark}{Remark}

\lstdefinestyle{sh}{basicstyle=\ttfamily\small, numbers=none, frame=none,
        xleftmargin=1.5em, columns=fixed, keepspaces=true, basewidth=0.5em}

\title{The smallest square tileable by pairwise incomparable\\integer rectangles}
\author{\small George M. Georgiou\thanks{\small School of Computer Science and
Engineering, California State University, San Bernardino}\\[3pt]
\small \href{mailto:georgiou@csusb.edu}{georgiou@csusb.edu}}
\date{}

\begin{document}
\maketitle

\begin{abstract}
Croft, Falconer and Guy (\emph{Unsolved Problems in Geometry}, Problem~C5) exhibit a
tiling of the $27\times27$ square by eight pairwise incomparable integer rectangles and
remark that it is not known whether $27$ is the smallest side length of a square that can
be tiled by pairwise incomparable integer rectangles, no restriction being placed on the
number of tiles. Knuth has since stated, in Volume~4B (2022) of \emph{The Art of
Computer Programming} and in a 2019 draft of it, that $27$ is indeed the smallest. We give a complete, independently
reproducible proof: for every integer $n\le26$ and every $k\ge2$, the $n\times n$ square
admits no tiling by $k$ pairwise incomparable integer rectangles. The proof combines two
structural reductions with an exhaustive search over the $167\,538$ surviving candidate
tile sets, carried out by two independently written programs. The same search shows that
at least eight tiles are needed at $n=27$ and that the tile set of Croft, Falconer and Guy
is the only one that tiles this square; it does so in exactly two arrangements up to
symmetry. We also give explicit eight-tile tilings of the $n\times n$ square for every
$n\ge27$, and seven-tile tilings for every $n\ge34$. With the classical lower bounds (every
tiling of a rectangle by two or more pairwise incomparable rectangles uses at least seven,
and no square of integer side less than $34$ can be tiled by seven pairwise incomparable
integer rectangles), this determines the least number of tiles for every square: there is
no tiling for $n\le26$, eight tiles are needed for $27\le n\le33$, and seven for $n\ge34$.
The complete software, build instructions and output logs are included as ancillary
files.
\end{abstract}

\section{The problem}

\subsection*{Conventions}

Throughout, $n$ denotes a positive integer. An \emph{integer rectangle} is a rectangle
whose two \emph{side lengths} are positive integers; nothing is assumed about the
position or the orientation of the rectangle, and Observation~\ref{obs:axis} and
Lemma~\ref{lem:int} below show that axis-parallelism and integrality of the corners are
both automatic. We write a rectangle's dimensions as an ordered pair $(a,b)$ with
$a\le b$, so that a rectangle and its quarter-turn have the same symbol: ``fits inside''
always permits a quarter-turn, since only the sorted pair of side lengths matters.

Two rectangles are \emph{incomparable} if neither fits inside the other with sides
parallel. In terms of sorted dimensions, $(a_1,a_2)$ and $(b_1,b_2)$ with $a_1\le a_2$ and
$b_1\le b_2$ are incomparable exactly when
\[
  a_1 < b_1 \le b_2 < a_2
  \qquad\text{or}\qquad
  b_1 < a_1 \le a_2 < b_2 ,
\]
that is, when one of them is strictly narrower \emph{and} strictly longer than the other.

A \emph{tiling} of the $n\times n$ square is a partition of it into $k\ge2$ rectangles
with pairwise disjoint interiors whose union is the square. As in~\cite{CFG}, the side
length of the containing square and the two side lengths of every tile are integers;
nothing is assumed about where the tiles sit, which is what Lemma~\ref{lem:int}
establishes. (For real-valued side lengths the question has no content: any tiling
rescales to any size, so ``smallest square'' is meaningful only for integer data.)

Problem~C5 of~\cite{CFG} states no orientation hypothesis for the tiles --- the phrase
``with sides parallel'' there governs the comparison of two rectangles, not their
placement --- and none is needed, because in a tiling of a rectangle by rectangles
alignment is forced. We record this so that the theorem below is not read as solving a
restricted problem.

\Needspace{8\baselineskip}
\begin{obs}\label{obs:axis}
Let $R$ be a rectangle tiled by finitely many rectangles $T_1,\dots,T_k$ with pairwise
disjoint interiors. Then every $T_i$ has its sides parallel to those of $R$.
\end{obs}

\begin{proof}
Each tile $T$ determines a pair of perpendicular directions, which we record as its
\emph{class} $\theta(T)\in\mathbb{R}/(\pi/2)\mathbb{Z}$.

First, if $\partial T_i\cap\partial T_j$ contains a segment of positive length, then
$\theta(T_i)=\theta(T_j)$: that segment lies on an edge of each tile, so the two tiles
have a common edge direction.

Second, the graph $G$ on $\{T_1,\dots,T_k\}$ joining two tiles when they share a segment
of positive length is connected. Suppose not, and split the tiles into two nonempty
groups with no such join; let $A$ and $B$ be the unions of the two groups, so that $A$
and $B$ are closed, nonempty, and $A\cup B=R$. If $T_i$ and $T_j$ lie in different
groups, then $T_i\cap T_j\subseteq\partial T_i\cap\partial T_j$ because the interiors are
disjoint, so it contains no segment of positive length by assumption; being convex, as an
intersection of convex sets, it is therefore empty or a single point. As there
are finitely many pairs, $F:=A\cap B$ is finite. Put $X=R\setminus F$. Then
$A\cap X$ and $B\cap X$ are disjoint, nonempty (each group has a tile, and a tile has
interior points, which cannot lie in $F$), closed in $X$, and cover $X$; so $X$ is
disconnected. But a rectangle minus a finite set of points is connected. Contradiction.

Third, some tile is aligned with $R$. The bottom edge of $R$ is a segment of positive
length covered by the finitely many sets $T_i\cap(\text{bottom edge})$, each of which is
a point or a segment because $T_i$ is convex; so some $T_i$ meets the bottom edge in a
segment of positive length. That segment lies in $\partial T_i$, as $T_i\subseteq R$,
hence $T_i$ has an edge along the bottom edge of $R$ and $\theta(T_i)=\theta(R)$.

By the first two steps $\theta$ is constant on $G$, and by the third that constant is
$\theta(R)$.
\end{proof}

We may therefore take all rectangles to be axis-parallel from now on, and do so without
further comment.

\subsection*{Background}

Reingold asked, as Problem~E2422 in the \emph{American Mathematical Monthly}~\cite{E2422},
whether a rectangle can be tiled by pairwise incomparable rectangles at all: ``Prove or
disprove: No rectangular region can be tiled with mutually incomparable rectangles.''
Two solutions were published~\cite{E2422sol}. Van~Emde Boas and Lenstra exhibited
an incomparable tiling of the $38\times10$ rectangle by seven integer rectangles, and
stated that at least seven rectangles are needed, even under the weaker requirement that
no tile fit inside another by a translation alone; of the proof their solution says only
that it ``is by trial and error''. Nuij showed that every rectangle can be tiled by
pairwise incomparable rectangles with real side lengths; his construction uses seven tiles
of equal area, except for finitely many aspect ratios, where it adds a thin eighth tile of
smaller area. Yao, Reingold and Sands~\cite{YRS} then took the quantitative question
further: they showed that among tilings by seven incomparable integer rectangles, the
$13\times22$ rectangle is the smallest, in both area and perimeter (see
also~\cite[p.~22]{Gardner}; for the perimeter, \cite[answer to exercise~375(a)]{Knuth}),
and that the $34\times34$ square is the smallest integer square
tileable by seven. (Wells~\cite{Wells} records the resulting bounds $7\le k\le8$ on the
least number of tiles; so does the corresponding MathWorld entry~\cite{MathWorld}.)
Gardner's 1979 account of this work~\cite[pp.~22--23]{Gardner} states that the least
number of incomparable rectangles that tile a larger rectangle is seven, and adds that a square
can be tiled by seven incomparable integer rectangles exactly when its side is at least $34$, and
that eight can tile the $27\times27$ square, ``the smallest one known'', though not proved
to be the smallest. Grosek's question for cuboids~\cite[Problem~C5]{CFG} and Jepsen's work on it~\cite{Jepsen}
carried the subject into three dimensions.
Croft, Falconer and Guy~\cite[Problem~C5, pp.~85--87]{CFG} display a $27\times27$ square
tiled by eight incomparable rectangles and write:

\begin{quote}
``\dots though Figure C6(c) shows a smaller $27\times27$ square using $8$ such tiles; it
is not known whether this is the smallest square possible for any number of incomparable
rectangles.''
\end{quote}

\noindent Knuth later answered this question in Section~7.2.2.1 of \emph{The Art of
Computer Programming}, published in Volume~4B~\cite[exercise~375 and its answer]{Knuth} and
available in draft by 2019~\cite{KnuthDraft}: he analyses the incomparable dissections into
seven and into eight rectangles, gives the parameters of the $27\times27$ solution, states
that no other choice achieves side $27$, and remarks that no smaller square can be dissected because nine rectangles would be too many.
(Exercise~374(d) there, with its answer, also proves that at least seven rectangles are
needed.) The theorem below confirms his answer with a complete proof
that covers every number of tiles, resting on an exhaustive and reproducible computation;
see Section~\ref{sec:priority}.

\begin{theorem}\label{thm:main}
Let $1\le n\le 26$ be an integer. Then for every $k\ge 2$ the $n\times n$ square admits no
tiling by $k$ pairwise incomparable integer rectangles. Consequently the $27\times27$
square is the smallest integer square that can be tiled by pairwise incomparable integer
rectangles.
\end{theorem}

(The case $n=1$ is trivial: the unit square cannot be split into two or more integer
rectangles at all. The search below begins at $n=2$.)

Thus the tiling of Figure~C6(c) of~\cite{CFG} is optimal, as Knuth stated~\cite{Knuth}. Figure~\ref{fig:27}(a) reproduces that tiling; the search of Section~\ref{sec:search}
rediscovers its tile set from scratch. Moreover eight tiles is the least possible number
at $n=27$ (Proposition~\ref{prop:min27}), and the tile set of Figure~\ref{fig:27} is the
only one that tiles this square, as Knuth also states, in exactly two arrangements up to
symmetry (Proposition~\ref{prop:unique27}), both drawn in Figure~\ref{fig:27}. Stretching the
arrangement of Figure~\ref{fig:27}(a) gives
an eight-tile tiling of every larger square; with the classical bounds
of~\cite{E2422sol,YRS,Knuth} this determines the least number of incomparable tiles for every
square (Corollary~\ref{cor:fewest}).

\begin{figure}[t]
\centering
\begin{minipage}[t]{0.48\textwidth}\centering
\begin{tikzpicture}[scale=0.260]
  \tikzset{tile/.style={draw=black, line width=0.5pt, fill=blue!8}}
  \draw[tile, fill=blue!25] (0,0) rectangle ++(1,27);
  \draw[tile, fill=blue!25] (1,0) rectangle ++(3,23);
  \draw[tile] (1,23) rectangle ++(21,4);
  \draw[tile] (22,8) rectangle ++(5,19);
  \draw[tile] (4,17) rectangle ++(18,6);
  \draw[tile] (4,0) rectangle ++(7,17);
  \draw[tile] (11,0) rectangle ++(16,8);
  \draw[tile] (11,8) rectangle ++(11,9);
  \node[rotate=90, font=\tiny] at (0.5,13.5) {$1\times27$};
  \node[rotate=90, font=\scriptsize] at (2.5,11.5) {$3\times23$};
  \node[font=\scriptsize] at (11.5,25.0) {$4\times21$};
  \node[rotate=90, font=\small] at (24.5,17.5) {$5\times19$};
  \node[font=\small] at (13.0,20.0) {$6\times18$};
  \node[rotate=90, font=\small] at (7.5,8.5) {$7\times17$};
  \node[font=\small] at (19.0,4.0) {$8\times16$};
  \node[font=\small] at (16.5,12.5) {$9\times11$};
  \draw[line width=1.1pt] (0,0) rectangle (27,27);
  \node[below] at (13.5,-0.3) {$27$};
  \node[left]  at (-0.3,13.5) {$27$};
\end{tikzpicture}
\\[0.3em] {\footnotesize (a) the tiling of Croft, Falconer and Guy}
\end{minipage}\hfill
\begin{minipage}[t]{0.48\textwidth}\centering
\begin{tikzpicture}[scale=0.260]
  \tikzset{tile/.style={draw=black, line width=0.5pt, fill=blue!8}}
  \draw[tile, fill=blue!25] (0,0) rectangle ++(1,27);
  \draw[tile] (1,0) rectangle ++(16,8);
  \draw[tile] (1,8) rectangle ++(5,19);
  \draw[tile] (6,8) rectangle ++(11,9);
  \draw[tile] (6,17) rectangle ++(18,6);
  \draw[tile] (6,23) rectangle ++(21,4);
  \draw[tile] (17,0) rectangle ++(7,17);
  \draw[tile, fill=blue!25] (24,0) rectangle ++(3,23);
  \node[rotate=90, font=\tiny] at (0.5,13.5) {$1\times27$};
  \node[font=\small] at (9.0,4.0) {$8\times16$};
  \node[rotate=90, font=\small] at (3.5,17.5) {$5\times19$};
  \node[font=\small] at (11.5,12.5) {$9\times11$};
  \node[font=\small] at (15.0,20.0) {$6\times18$};
  \node[font=\scriptsize] at (16.5,25.0) {$4\times21$};
  \node[rotate=90, font=\small] at (20.5,8.5) {$7\times17$};
  \node[rotate=90, font=\scriptsize] at (25.5,11.5) {$3\times23$};
  \draw[line width=1.1pt] (0,0) rectangle (27,27);
  \node[below] at (13.5,-0.3) {$27$};
  \node[left]  at (-0.3,13.5) {$27$};
\end{tikzpicture}
\\[0.3em] {\footnotesize (b) the second arrangement}
\end{minipage}
\caption{The two tilings of the $27\times27$ square by eight pairwise incomparable
rectangles, up to the symmetries of the square; both use the same tile set
(Proposition~\ref{prop:unique27}). (a) The tiling of Figure~C6(c) of~\cite{CFG},
recovered by the search of Section~\ref{sec:search}. (b) The second arrangement, obtained
from (a) by reflecting the $26\times27$ block to the right of the $1\times27$ strip. The
shaded tiles, $1\times27$ and $3\times23$, share an edge in (a) but not in (b). By
Theorem~\ref{thm:main} no smaller square admits such a tiling, and by
Proposition~\ref{prop:min27} no tiling of this square uses fewer tiles.}
\label{fig:27}
\end{figure}

\FloatBarrier
\section{Two reductions}

\subsection{The staircase lemma}

\begin{lemma}\label{lem:stair}
Let $T_1,\dots,T_k$ be pairwise incomparable rectangles with sorted dimensions
$(a_i,b_i)$, $a_i\le b_i$. Then the $a_i$ are pairwise distinct, the $b_i$ are pairwise
distinct, and if the tiles are indexed so that $a_1<a_2<\cdots<a_k$ then necessarily
\[
  b_1>b_2>\cdots>b_k .
\]
\end{lemma}

\begin{proof}
Incomparability says exactly that the pairs $(a_i,b_i)$ form an \emph{antichain} in the
product order on $\mathbb{Z}^2$: no pair dominates another coordinatewise. If $a_i=a_j$
for $i\ne j$, then one of $b_i\le b_j$, $b_j\le b_i$ holds, and the corresponding pair is
dominated, so $T_i$ and $T_j$ would be comparable; hence the $a_i$ are distinct, and
symmetrically so are the $b_i$. Having indexed so that the $a_i$ increase, suppose
$b_i<b_j$ for some $i<j$; then $(a_i,b_i)$ is dominated by $(a_j,b_j)$, a contradiction.
So the $b_i$ decrease, and strictly, since they are distinct.
\end{proof}

Two consequences bound the search space. Both use that all quantities are integers.

\begin{corollary}[side bound]\label{cor:side}
$b_1\ge2k-1$; and since every tile fits inside the square, $n\ge2k-1$.
\end{corollary}

\begin{proof}
The $a_i$ are $k$ distinct positive integers with $a_1<\cdots<a_k$, so $a_k\ge a_1+(k-1)
\ge k$. The $b_i$ are $k$ distinct positive integers with $b_1>\cdots>b_k$, so
$b_1\ge b_k+(k-1)$; each of the $k-1$ strict integer inequalities $b_1>b_2$, \dots,
$b_{k-1}>b_k$ contributes at least $1$. Finally $b_k\ge a_k\ge k$, whence
$b_1\ge k+(k-1)=2k-1$. Every tile fits in the square, so $n\ge b_1\ge2k-1$.
\end{proof}

\begin{corollary}[area bound]\label{cor:area}
With the indexing of Lemma~\ref{lem:stair}, $a_i\ge i$ and $b_i\ge2k-i$ for every $i$, so
\[
  n^2 \;=\; \sum_{i=1}^{k} a_i b_i
      \;\ge\; \sum_{i=1}^{k} i\,(2k-i)
      \;=\; \frac{k(k+1)(4k-1)}{6}.
\]
For $k=10$ the right-hand side is $715>676=26^2$. Hence any incomparable tiling of an
$n\times n$ square with $n\le26$ uses at most nine tiles.
\end{corollary}

\begin{proof}
As in Corollary~\ref{cor:side}, $a_i\ge a_1+(i-1)\ge i$ and $b_i\ge b_k+(k-i)\ge
a_k+(k-i)\ge k+(k-i)=2k-i$. Summing $i(2k-i)$ for $i=1,\dots,k$ gives
$2k\cdot\frac{k(k+1)}{2}-\frac{k(k+1)(2k+1)}{6}=\frac{k(k+1)(4k-1)}{6}$, which is
increasing in $k$.
\end{proof}

\begin{remark}
We do \emph{not} appeal to the general lower bound $k\ge7$ of~\cite{E2422sol,YRS,Knuth}; the
search below includes $2\le k\le6$ and finds no solution in the finite range considered
here. This is a much weaker statement than that bound, which concerns incomparable tilings
of arbitrary rectangles of any dimensions, and nothing here reproves it.
\end{remark}

\subsection{Integrality of placements}

\begin{lemma}\label{lem:int}
Let $R$ be a rectangle with integer corners, tiled by rectangles each of whose side
lengths is an integer. Then every tile has integer corners.
\end{lemma}

\begin{proof}
Place $R=[0,W]\times[0,H]$ with $W,H\in\mathbb{Z}$. Let $S$ be a vertical side of a tile,
lying on the line $x=c$. Only finitely many horizontal lines contain a horizontal side of
a tile, so we may choose a height $y$ in the relative interior of $S$ that avoids all of
them. Let $\ell$ be the horizontal line at height $y$. Every tile meeting $\ell$ meets it
in the interior of its vertical extent, hence meets $\ell$ in a segment equal to the
tile's full width; these segments have pairwise disjoint interiors and cover $\ell\cap R$,
so they partition $[0,W]$ into consecutive intervals whose lengths are the widths of the
corresponding tiles, and are therefore integers. Starting from the integer left endpoint
$0$, every endpoint of this partition is an integer. The point $c$ is such an endpoint,
being an endpoint of the cross-section of the tile with side $S$; hence $c\in\mathbb{Z}$.

Applying the same argument to vertical lines shows that every horizontal tile side lies at
an integer height. Every tile corner is the intersection of a vertical and a horizontal
tile side, so all corners are integer points.
\end{proof}

Consequently nothing is lost by searching on the unit grid: any tiling of the $n\times n$
square by integer rectangles is a partition of the $n^2$ unit cells.

\section{The search}\label{sec:search}

Fix $n$. The search has two stages, each of which is exhaustive.

\subsection{Stage 1: enumerating candidate tile sets}

By Lemma~\ref{lem:stair}, and since every tile fits inside the square, a set of tiles that
could possibly occur is described by integers
\begin{equation}\label{eq:staircase}
  1\le a_1<a_2<\cdots<a_k,\qquad
  n\ge b_1>b_2>\cdots>b_k,\qquad
  a_i\le b_i\ \ (1\le i\le k),\qquad
  \sum_{i=1}^{k}a_ib_i=n^2 ,
\end{equation}
with $k\ge2$. Note that all $k$ tiles are distinct, since two congruent rectangles are
comparable; so a tile set is a set, not a multiset.

The primary program enumerates all solutions of~\eqref{eq:staircase} by depth-first
search, appending tiles in the order $i=1,2,\dots$; at a node with prefix
$(a_1,b_1),\dots,(a_t,b_t)$ of area $A=\sum_{i\le t}a_ib_i$ it writes $\alpha=a_t$ and
$\beta=b_t$ (with $\alpha=0$, $\beta=n+1$ at the root) and iterates $a$ over
$\alpha<a\le\beta-1$ and, for each such $a$, $b$ over $a\le b\le\beta-1$. Three pruning
rules are applied. They are the only ones, and each is safe.

\begin{proposition}\label{prop:prune}
Let $r=n^2-A$ denote the residual area at a node.
\begin{enumerate}
\item[\textnormal{(P1)}] If $r=0$ the node is a complete tile set and is not extended.
\item[\textnormal{(P2)}] The $a$-loop is terminated at the first $a$ with $a^2>r$.
\item[\textnormal{(P3)}] For fixed $a$, the $b$-loop is terminated at the first $b$ with
$ab>r$.
\end{enumerate}
No solution of~\eqref{eq:staircase} extending the current prefix is lost.
\end{proposition}

\begin{proof}
(P1) All tiles have area $\ge1$, so extending a prefix strictly increases $A$; a prefix of
area exactly $n^2$ admits no completion other than itself.
(P2) The next tile has $b\ge a$, so its area is at least $a^2$; if $a^2>r$ its area
exceeds the residual, and since all subsequent tiles again have positive area no
completion exists. As $a$ increases along the loop the same holds for every larger $a$, so
terminating the loop is safe.
(P3) For fixed $a$ the loop variable $b$ increases, so $ab>r$ persists for all larger $b$.
\end{proof}

The implementation also carries a nominal cap $k\le K_{\max}$; the runs supporting
Theorem~\ref{thm:main} use $K_{\max}=40$, which by Corollary~\ref{cor:area} exceeds the
largest possible $k$ ($\le9$) in the range $n\le26$, so the cap never binds. Setting
$K_{\min}=K_{\max}=k$ restricts the enumeration to sets of exactly $k$ tiles; this is how
Proposition~\ref{prop:min27} is computed, from the same program.

\begin{remark}
An earlier draft of the program carried a fourth rule, pruning on an \emph{upper} bound
for the area attainable by the remaining tiles. That bound was wrong and discarded valid
candidate sets. The rule is absent from both programs used here; see
Section~\ref{sec:verif}(v).
\end{remark}

\subsection{Stage 2: deciding tileability of a candidate set}

By Lemma~\ref{lem:int} this is a finite problem on the $n\times n$ grid of unit cells,
solved by backtracking on the lowest, then leftmost, empty cell. Completeness is the
following standard observation.

\begin{lemma}\label{lem:bl}
Let a partial placement of some of the tiles be given, and let $c$ be the empty cell in the
lowest non-full row and, within that row, in the leftmost empty column. In any completion
of the partial placement to a tiling, the tile covering $c$ has $c$ as its lower-left cell.
\end{lemma}

\begin{proof}
All cells in rows below $c$ are filled, so the tile covering $c$ cannot extend below $c$;
and all cells to the left of $c$ in its own row are filled, so it cannot extend to the
left of $c$.
\end{proof}

Hence, enumerating every unused tile in each of its two orientations with lower-left cell
at $c$ (and discarding those that do not fit inside the square or overlap a placed tile) is
exhaustive; recursing produces every tiling. If the recursion exhausts all branches, no
tiling exists.

The converse --- soundness of the acceptance test --- is equally short, and we record it
because the program's base case depends on it. Suppose the recursion has placed all $k$
tiles of the candidate set inside the square with pairwise disjoint interiors. Their union
$U$ is a finite union of closed rectangles, hence closed, and has area
$\sum_i a_ib_i=n^2$ by~\eqref{eq:staircase}, the tiles being disjoint. If $U$ were not the
whole square $S$, then $S\setminus U$ would be a nonempty relatively open subset of $S$
and would therefore have positive area, forcing $\operatorname{area}(S)>n^2$, a
contradiction. (Equality of areas alone would not suffice: a set of measure zero could be
missing. Closedness of $U$ is what rules that out.) So $U=S$ and the placement is a
tiling. Accepting as soon as every tile has been placed is therefore correct, and no
separate check that no cell was left empty is needed.

\subsection{Results}

\begin{table}[H]
\centering
\begin{tabular}{@{}rrl@{\qquad}rrl@{}}
\toprule
$n$ & candidate tile sets & tileable & $n$ & candidate tile sets & tileable\\
\midrule
$2$--$6$ & $0$ & no      & $17$ & $392$    & no\\
$7$      & $1$ & no      & $18$ & $732$    & no\\
$8$      & $1$ & no      & $19$ & $1\,278$ & no\\
$9$      & $2$ & no      & $20$ & $2\,311$ & no\\
$10$     & $4$ & no      & $21$ & $4\,179$ & no\\
$11$     & $8$ & no      & $22$ & $7\,474$ & no\\
$12$     & $18$ & no     & $23$ & $13\,256$ & no\\
$13$     & $33$ & no     & $24$ & $23\,344$ & no\\
$14$     & $67$ & no     & $25$ & $41\,326$ & no\\
$15$     & $116$ & no    & $26$ & $72\,777$ & no\\
$16$     & $219$ & no    & $\mathbf{27}$ & $127\,295$ & \textbf{yes} (one set)\\
\bottomrule
\end{tabular}
\caption{Outcome of the exhaustive search; $167\,538$ candidate tile sets in total for
$n\le26$, none of which tiles its square. Both implementations of
Section~\ref{sec:verif} produce every candidate-set count and every negative result for
$2\le n\le26$. At $n=27$ both find $127\,295$ candidate sets, exactly one of which tiles
(Proposition~\ref{prop:unique27}).}
\label{tab:res}
\end{table}

Table~\ref{tab:res} records the outcome; the negative entries prove
Theorem~\ref{thm:main}. The first success occurs at $n=27$, with the tile set
\[
  \{\,1\times27,\;3\times23,\;4\times21,\;5\times19,\;6\times18,\;7\times17,\;8\times16,
     \;9\times11\,\},
\]
which is precisely the tile set of Figure~C6(c) of~\cite{CFG}; the placement found is the
one drawn in Figure~\ref{fig:27}(a).

\begin{proposition}\label{prop:min27}
Every incomparable tiling of the $27\times27$ square uses at least eight tiles.
\end{proposition}

\begin{proof}
Restricting the enumeration to exactly $k$ tiles exhausts $10$, $251$, $3\,183$,
$16\,527$, $39\,442$ and $44\,033$ candidate sets for $k=2,\dots,7$ respectively, and none
of them tiles. This is the program of Section~\ref{sec:search} with
$K_{\min}=K_{\max}=k$:
\begin{lstlisting}[style=sh]
for k in 2 3 4 5 6 7; do ./incomp 27 27 1 "$k" "$k"; done
\end{lstlisting}
(output in \texttt{logs/n27\_k2\_k7.log}).
\end{proof}

\begin{proposition}\label{prop:unique27}
Among the $127\,295$ candidate tile sets for $n=27$, exactly one tiles the square: the set
of Figure~\ref{fig:27}. It tiles the square in exactly sixteen ways, which form two classes
of eight under the symmetries of the square: the classes of Figure~\ref{fig:27}(a)
and~(b).
\end{proposition}

\begin{proof}
The first statement is the full run of both programs at $n=27$,
\begin{lstlisting}[style=sh]
./incomp 27 27 1 2 40 ;  ./verify2 27 27
\end{lstlisting}
which agree on the count and each report a single tileable set (output in
\texttt{logs/n27\_all\_k.log}). The two searches stop at the first tiling of each set, so the
count of tilings comes from a third program, \texttt{arrangements.py}, which lists every
tiling of a given tile set, by the backtracking of Lemma~\ref{lem:bl} recording every
completion, and groups them under the eight symmetries of the square (output in
\texttt{logs/arrangements\_27.log}). That the two classes differ can also be seen by hand.
Reflecting the $26\times27$ block to the right of the $1\times27$ strip of
Figure~\ref{fig:27}(a) in its vertical axis gives the tiling of Figure~\ref{fig:27}(b),
\[
\begin{aligned}
&(0,0;1,27),\quad(1,0;16,8),\quad(1,8;5,19),\quad(6,8;11,9),\\
&(6,17;18,6),\quad(6,23;21,4),\quad(17,0;7,17),\quad(24,0;3,23),
\end{aligned}
\]
written as $(x,y;w,h)$ for a tile with lower-left corner $(x,y)$, width $w$ and height $h$.
In Figure~\ref{fig:27}(a) the $1\times27$ and $3\times23$ tiles share an edge; in
Figure~\ref{fig:27}(b) they do not touch. A symmetry of the square carrying one tiling to
the other would map each of
these two tiles to itself, each being the only tile of its size, and would preserve
contact between tiles; so no such symmetry exists.
\end{proof}

\subsection{Larger squares and the least number of tiles}\label{sec:larger}

The arrangement of Figure~\ref{fig:27}(a) stretches to every larger square, and a seven-tile
arrangement of the $34\times34$ square stretches in the same way. As above, $(x,y;w,h)$
denotes a tile with lower-left corner $(x,y)$, width $w$ and height $h$, the origin being
the lower-left corner of the square. We use the converse of Lemma~\ref{lem:stair}, which is
immediate: if $a_1<\cdots<a_k$ and $b_1>\cdots>b_k$, then for $i<j$ the tile $T_i$ is
strictly narrower and strictly longer than $T_j$, so the tiles are pairwise incomparable.

\begin{proposition}\label{prop:eight}
For every integer $n\ge27$ the eight rectangles
\[
\begin{aligned}
&(0,0;1,n),\quad(1,0;3,n-4),\quad(1,n-4;n-6,4),\quad(n-5,8;5,n-8),\\
&(4,n-10;n-9,6),\quad(4,0;7,n-10),\quad(11,0;n-11,8),\quad(11,8;n-16,n-18)
\end{aligned}
\]
tile the $n\times n$ square and are pairwise incomparable. At $n=27$ this is the tiling of
Figure~\ref{fig:27}(a).
\end{proposition}

\begin{proof}
\emph{Covering.} Read the square column by column. Over $0\le x\le1$ lies the
$1\times n$ tile. Over $1\le x\le4$: the tile $[1,4]\times[0,n-4]$, then the top band
$[1,n-5]\times[n-4,n]$. Over $4\le x\le11$: the tile $[4,11]\times[0,n-10]$, the band
$[4,n-5]\times[n-10,n-4]$ and the top band. Over $11\le x\le n-5$: the bottom band
$[11,n]\times[0,8]$, the tile $[11,n-5]\times[8,n-10]$ and the two bands above it. Over
$n-5\le x\le n$: the bottom band, then the tile $[n-5,n]\times[8,n]$. In each column the
listed tiles stack up to the full height and meet only along their boundaries, because
$n-10\ge8$ and $n-5\ge11$.
\emph{Incomparability.} The sorted dimensions are
\[
(1,n),\ (3,n-4),\ (4,n-6),\ (5,n-8),\ (6,n-9),\ (7,n-10),\ (8,n-11),\ (n-18,n-16).
\]
The short sides increase strictly exactly when $n-18>8$, that is $n\ge27$, and the long sides
always decrease strictly.
\end{proof}

\begin{proposition}\label{prop:seven}
For every integer $n\ge34$ the seven rectangles
\[
\begin{aligned}
&(0,0;3,n-6),\quad(3,0;10,n-14),\quad(13,0;n-13,9),\quad(13,9;n-20,n-23),\\
&(n-7,9;7,n-9),\quad(3,n-14;n-10,8),\quad(0,n-6;n-7,6)
\end{aligned}
\]
tile the $n\times n$ square and are pairwise incomparable.
\end{proposition}

\begin{proof}
\emph{Covering.} Over $0\le x\le3$: the tile $[0,3]\times[0,n-6]$, then the top band
$[0,n-7]\times[n-6,n]$. Over $3\le x\le13$: the tile $[3,13]\times[0,n-14]$, the band
$[3,n-7]\times[n-14,n-6]$ and the top band. Over $13\le x\le n-7$: the bottom band
$[13,n]\times[0,9]$, the tile $[13,n-7]\times[9,n-14]$ and the two bands above it. Over
$n-7\le x\le n$: the bottom band, then the tile $[n-7,n]\times[9,n]$. In each column the
listed tiles stack up to the full height and meet only along their boundaries, because
$n-14\ge9$ and $n-7\ge13$.
\emph{Incomparability.} The sorted dimensions are
\[
(3,n-6),\ (6,n-7),\ (7,n-9),\ (8,n-10),\ (9,n-13),\ (10,n-14),\ (n-23,n-20),
\]
a strict staircase exactly when $n-23>10$, that is $n\ge34$.
\end{proof}

\begin{figure}[t]
\centering
\begin{minipage}[t]{0.48\textwidth}\centering
\begin{tikzpicture}[scale=0.180]
  \tikzset{tile/.style={draw=black, line width=0.5pt, fill=blue!8}}
  \draw[tile] (0,0) rectangle ++(1,36);
  \draw[tile] (1,0) rectangle ++(3,32);
  \draw[tile] (1,32) rectangle ++(30,4);
  \draw[tile] (31,8) rectangle ++(5,28);
  \draw[tile] (4,26) rectangle ++(27,6);
  \draw[tile] (4,0) rectangle ++(7,26);
  \draw[tile] (11,0) rectangle ++(25,8);
  \draw[tile] (11,8) rectangle ++(20,18);
  \node[rotate=90, font=\scriptsize] at (2.5,16.0) {$3\times(n-4)$};
  \node[font=\scriptsize] at (16.0,34.0) {$4\times(n-6)$};
  \node[rotate=90, font=\scriptsize] at (33.5,22.0) {$5\times(n-8)$};
  \node[font=\scriptsize] at (17.5,29.0) {$6\times(n-9)$};
  \node[rotate=90, font=\scriptsize] at (7.5,13.0) {$7\times(n-10)$};
  \node[font=\scriptsize] at (23.5,4.0) {$8\times(n-11)$};
  \node[font=\scriptsize] at (21.0,17.0) {$(n-18)\times(n-16)$};
  \draw[line width=1.1pt] (0,0) rectangle (36,36);
\end{tikzpicture}
\\[0.3em] {\footnotesize (a) $n\ge27$, eight tiles ($n=36$)}
\end{minipage}\hfill
\begin{minipage}[t]{0.48\textwidth}\centering
\begin{tikzpicture}[scale=0.162]
  \tikzset{tile/.style={draw=black, line width=0.5pt, fill=blue!8}}
  \draw[tile] (0,0) rectangle ++(3,34);
  \draw[tile] (3,0) rectangle ++(10,26);
  \draw[tile] (13,0) rectangle ++(27,9);
  \draw[tile] (13,9) rectangle ++(20,17);
  \draw[tile] (33,9) rectangle ++(7,31);
  \draw[tile] (3,26) rectangle ++(30,8);
  \draw[tile] (0,34) rectangle ++(33,6);
  \node[rotate=90, font=\scriptsize] at (1.5,17.0) {$3\times(n-6)$};
  \node[rotate=90, font=\scriptsize] at (8.0,13.0) {$10\times(n-14)$};
  \node[font=\scriptsize] at (26.5,4.5) {$9\times(n-13)$};
  \node[font=\tiny] at (23.0,17.5) {$(n-23)\times(n-20)$};
  \node[rotate=90, font=\scriptsize] at (36.5,24.5) {$7\times(n-9)$};
  \node[font=\scriptsize] at (18.0,30.0) {$8\times(n-10)$};
  \node[font=\scriptsize] at (16.5,37.0) {$6\times(n-7)$};
  \draw[line width=1.1pt] (0,0) rectangle (40,40);
\end{tikzpicture}
\\[0.3em] {\footnotesize (b) $n\ge34$, seven tiles ($n=40$)}
\end{minipage}
\caption{(a) The eight-tile family of Proposition~\ref{prop:eight}, drawn at $n=36$; the
unlabelled strip on the left is $1\times n$. At $n=27$ it is the tiling of
Figure~\ref{fig:27}(a). (b) The seven-tile family of Proposition~\ref{prop:seven}, drawn
at $n=40$. Each tile is labelled by its sorted dimensions $a\times b$, as functions
of~$n$.}
\label{fig:larger}
\end{figure}

\begin{corollary}\label{cor:fewest}
Let $m(n)$ be the least $k\ge2$ for which the $n\times n$ square has a tiling by $k$
pairwise incomparable integer rectangles. There is no such tiling for $n\le26$;
$m(n)=8$ for $27\le n\le33$; and $m(n)=7$ for $n\ge34$. In particular, the squares that
can be tiled by exactly seven incomparable rectangles are those with $n\ge34$, and those
that can be tiled by exactly eight are those with $n\ge27$.
\end{corollary}

\begin{proof}
The case $n\le26$ is Theorem~\ref{thm:main}. Propositions~\ref{prop:eight}
and~\ref{prop:seven} give the upper bounds and the two descriptions of tileable sizes. For
the lower bounds, six or fewer incomparable rectangles never tile a rectangle (stated by
van~Emde Boas and Lenstra~\cite{E2422sol}, proved by Knuth~\cite[exercise~374(d) and its
answer]{Knuth}; see also~\cite{YRS} and~\cite[p.~22]{Gardner}), and seven never tile a
square of side less than $34$~\cite{YRS} (see also~\cite[answer to exercise~375(a)]{Knuth}
and~\cite[p.~22]{Gardner}).
\end{proof}

Unlike Theorem~\ref{thm:main}, Corollary~\ref{cor:fewest} relies on this general bound of
seven, and for $28\le n\le33$ also on the bound for seven tiles. The 1974
statement~\cite{E2422sol} of the bound of seven gives no proof beyond ``trial and error'';
Knuth's proof~\cite[answer to exercise~374(d)]{Knuth} rests on his analysis of reduced
patterns and motley dissections (exercises~365 and~367 there), which we have not
re-verified. As a check of the bound for seven tiles, \texttt{incomp} with exactly
seven tiles finds no tiling of any square of side at most $33$
(\texttt{logs/k7\_squares.log} for $n\le27$ and \texttt{logs/k7\_squares\_28\_33.log} for
$28\le n\le33$).

The first version of this paper listed ten-tile tilings of the squares of side $29$, $30$
and $31$, the first ones found by \texttt{incompf}. They remain valid and are reproduced
by \texttt{make witness} (\texttt{logs/witnesses\_27\_31.log}), but
Proposition~\ref{prop:eight} supersedes them; the tiling it found at $n=28$ is
Proposition~\ref{prop:eight} at $n=28$. We have not checked whether the construction of
Proposition~\ref{prop:eight} appears in earlier literature.

We include Proposition~\ref{prop:seven} for completeness only. Gardner, reporting on the
work of Reingold, Yao and Sands, already states that ``a square can be tiled with seven
incomparable rectangles having integral sides if and only if its side is 34 or
larger''~\cite[p.~22]{Gardner}. We do not know whether the coordinates of
Proposition~\ref{prop:seven} appear in~\cite{YRS}, and we claim no novelty for
Proposition~\ref{prop:seven} or for the case $k=7$ of Corollary~\ref{cor:fewest}.
Knuth~\cite[answer to exercise~375(a)]{Knuth} finds the side $34$ by linear programming
over the possible orderings of the dimensions, gives the parameters of the $34\times34$
seven-rectangle solution and notes that it is unique.

\section{Reproducibility and validation}\label{sec:verif}

\subsection{The artifact}

Theorem~\ref{thm:main} depends on a computation, so the software, the exact commands and
the full output logs are archived with this paper, under the MIT licence, as the single
file
\begin{center}
\texttt{incomparable-square-artifact.tar.gz},\qquad
SHA-256 \;\texttt{38339843\,04584b47\,38b961ad\,bb4b27fa}\\
\hspace{5.1cm}\texttt{8f63ca4b\,37a39e53\,ed3281e0\,3c9ecd88}
\end{center}
The archive is attached to the arXiv version of this paper as an ancillary file, and is
also available from the author. It unpacks to a single directory containing:

\begin{itemize}
\item \texttt{incomp.c} --- the primary search of Section~\ref{sec:search}, which also
performs the fixed-$k$ runs behind Proposition~\ref{prop:min27}. \\
SHA-256 \texttt{a96a67ab\dots77d2ce4f}.
\item \texttt{verify2.c} --- the independent re-implementation described below. \\
SHA-256 \texttt{37532834\dots4059c9cc}.
\item \texttt{incompf.c} --- an existence-only variant of \texttt{incomp.c}, differing
from it \emph{only} by the guard \texttt{if (found\_any) return;} at the head of
\texttt{dfs()}, which stops the enumeration at the first tiling found. Because it does not
complete the enumeration it is used only to produce the witnesses mentioned in
Section~\ref{sec:larger}, never for a negative result.
\item \texttt{verify.py} --- a third implementation, in Python, which doubles as a
certificate checker for explicit tile sets.
\item \texttt{roundtrip.py} --- the randomised test of Section~\ref{sec:verif}(iv).
\item \texttt{arrangements.py} --- enumerates every tiling of a given tile set, groups the
tilings under the symmetries of the square and prints them all
(Proposition~\ref{prop:unique27}).
\item \texttt{expect.py} --- checks the output of every target of \texttt{make check} except
\texttt{figcheck}: that of the search runs, of \texttt{arrangements} and of
\texttt{stress} against the archived log (for the searches: square sizes; numbers of
candidate, tested and tileable sets; verdicts), and that of \texttt{figures} by requiring
each of its four explicit tile sets to be reported tileable; it also asserts the claims of
this paper, and fails on any discrepancy.
\item \texttt{figures.py} --- verifies the coordinate certificates and the two families
of Section~\ref{sec:larger} (the latter for every $n$ from $27$, respectively $34$, up to
$120$), emits the TikZ of both panels of Figure~\ref{fig:27} from the
certificates, and audits the figures of this paper: in \texttt{check} mode it parses every
picture in the source and verifies that each drawn rectangle and each dimension label
matches its certificate or, in Figure~\ref{fig:larger}, the family evaluated at the
drawn size (\texttt{make figcheck}).
\item \texttt{Makefile}, \texttt{README.md}, \texttt{LICENSE}, \texttt{SHA256SUMS}, and
the directory \texttt{logs/} with the machine-readable output of every run quoted here.
\end{itemize}

\noindent
\texttt{SHA256SUMS} pins every file individually. The full source of the two
proof-critical programs is also reproduced in Appendices~\ref{app:incomp}
and~\ref{app:verify2}, so the computation behind Theorem~\ref{thm:main} can be checked
without the archive. The programs behind the supporting results, including
\texttt{arrangements.py} for the count of arrangements in
Proposition~\ref{prop:unique27}, are supplied only in the archive.

\paragraph{Environment.}
Ubuntu 22.04.5 LTS, Linux 6.8.0, x86-64; gcc 11.4.0 with \texttt{-O2 -Wall} (both
programs compile without warnings); CPython 3.10.12, standard library only; Intel Core
i5-1155G7 at 2.50\,GHz. All programs are serial, single-threaded and deterministic apart
from \texttt{roundtrip.py}, whose seed is an explicit command-line argument (seeds
$1,2,3,4$ were used). The runs added in the second version of this paper (the full run at
$n=27$ of Proposition~\ref{prop:unique27} and the seven-tile runs for $28\le n\le33$) were
made on Ubuntu 24.04.5 LTS with gcc 13.3.0 (\texttt{-O2 -Wall}, no warnings) and
CPython 3.13 on an Intel Xeon at 2.10\,GHz, as was the enumeration of arrangements in
Proposition~\ref{prop:unique27}; on that machine the runs \texttt{proof1} and
\texttt{proof2} reproduce Table~\ref{tab:res} exactly.

\paragraph{Supported input range.}
The bitmask representation permits side lengths up to $64$ --- row occupancy is a
\texttt{uint64\_t} with one bit per column --- and the tile arrays permit up to $63$ tiles
(\texttt{incomp} requires \texttt{k\_max < MAXTILES}, and both programs stop rather than
insert a $64$th tile),
far above the maximum permitted by Corollary~\ref{cor:area}: at most nine tiles for
$n\le26$, and at most eighteen even at $n=64$. Arguments are
validated and out-of-range input is rejected with exit status~$2$. We claim nothing about
behaviour near that limit, where the counters could in principle overflow and the
runtimes are in any case prohibitive: all results reported here that use the C programs
have $n\le33$, and the proof of Theorem~\ref{thm:main} uses $n\le26$.

\paragraph{Running the proof.}
\begin{lstlisting}[style=sh]
make                      # builds incomp, verify2, incompf
make proof1               # ./incomp  2 26 1 2 40          34.6 s
make proof2               # ./verify2 2 26                 46.6 s
\end{lstlisting}
Together these commands reproduce the computational verification used in the proof of
Theorem~\ref{thm:main}; the proof itself consists of the reductions of
Section~\ref{sec:search}, the two programs, and their verified execution. Both must print
\texttt{NO} for every $n$ from $2$ to $26$, and the candidate-set counts must agree entry
for entry. Their output is recorded in \texttt{logs/incomp\_squares\_2\_26.log} and
\texttt{logs/verify2\_squares\_2\_26.log}. The targets \texttt{prop2}, \texttt{n27all},
\texttt{arrangements}, \texttt{k7}, \texttt{k7ext}, \texttt{witness}, \texttt{figures},
\texttt{figcheck}, \texttt{python} and \texttt{stress} reproduce the supporting runs of
Sections~\ref{sec:search}--\ref{sec:verif}; \texttt{make check} runs everything except
\texttt{witness} and \texttt{k7ext}, in about four minutes. Each of its targets
except \texttt{figcheck} checks its output with \texttt{expect.py}: for the search runs,
every result line is compared with the archived log (square size; numbers of candidate,
tested and tileable sets; verdict), the expected range of sizes is checked, and the claims
of this paper are asserted directly; \texttt{arrangements} and \texttt{stress} are compared
with their logs, and \texttt{figures} must report every tile set tileable;
\texttt{figcheck} compares every figure with its certificate itself. Any discrepancy makes \texttt{make check} fail, so a successful run
means that the results agree with the paper. Since \texttt{make figcheck} reads the
paper source, the archive should be extracted next to it; its \texttt{README.md} gives the
commands.

\subsection{Validation}

The two programs of item~(i) are a genuine independence check on the computation. The
remaining items are supporting evidence --- useful for confidence, but not logically
independent proofs, and we do not present them as such.

\begin{enumerate}
\item[(i)] \textbf{Two independent programs.} \texttt{incomp.c} enumerates tile sets
through the staircase of Lemma~\ref{lem:stair} and tiles by filling the \emph{bottom-left}
free cell. \texttt{verify2.c} instead lists all $\binom{n+1}{2}$ candidate rectangles
lexicographically and selects them by increasing index, tests incomparability \emph{from
the definition} rather than via Lemma~\ref{lem:stair}, makes no ordering assumption in its
area test, fills the \emph{top-right} free cell (Lemma~\ref{lem:bl} applies in the
reflected form), and tries tiles in the reverse order. The two agree on the candidate-set
count and the answer for every $n$ from $2$ to $26$, and at $n=27$. Since
\texttt{verify2.c} does not assume Lemma~\ref{lem:stair}, the agreement of the counts is
also evidence for that lemma and its implementation. The tiling routines are less
independent than the enumerations: both store row occupancy as a 64-bit mask, and they use
mirror-image forms of the same corner rule.

\item[(ii)] \textbf{A third implementation.} \texttt{verify.py}, again definition-based,
reproduces the same counts and the same negative answers for $n\le17$; it is too slow
beyond that.

\item[(iii)] \textbf{Reproduction of published results.} All three tilings of Figure~C6
of~\cite{CFG} are confirmed by \texttt{verify.py} from the bare tile sets: the
$13\times22$ rectangle with seven tiles, the $34\times34$ square with seven, and the
$27\times27$ square with eight (\texttt{logs/certificates.log}). The $13\times22$ and
$27\times27$ tile sets were moreover \emph{found} by the search, matching~\cite{CFG}.
Restricting the search to exactly seven tiles returns ``no'' for every square of side at
most $33$, consistent with the result of~\cite{YRS} that $34$ is the smallest such square.

\item[(iv)] \textbf{Randomised completeness test.} $3\,641$ random tilings of squares ---
generated by random guillotine cuts together with random non-guillotine pinwheel
substitutions, with up to nine pieces and sides up to $26$ --- were stripped of their
positions and handed, as bare lists of dimensions (which may contain repeats in these
random examples), to the Python tiling oracle of \texttt{verify.py}. It recovered a tiling
in every case: no false negatives. This tests the Python oracle, whose corner rule is the
one used by the C programs; it is a diagnostic, not part of the proof.

\item[(v)] \textbf{A bug caught by cross-checking.} A draft version of the enumerator
carried an extra pruning rule based on an upper bound for the area of the remaining tiles.
The bound was incorrect --- it used the smallest admissible short side where the largest
was required --- and the rule silently discarded valid candidate sets, reporting
$21\,871$ sets at $n=24$ instead of $23\,344$. The discrepancy was detected only because
the candidate-set counts, and not merely the yes/no answers, are compared between
implementations; the rule was removed. It is absent from both programs used here. The only
pruning rules of \texttt{incomp.c} are those of Proposition~\ref{prop:prune};
\texttt{verify2.c}, which enumerates differently, skips a rectangle only when it would
overshoot the area (with \texttt{continue}, not \texttt{break}) or is comparable with one
already chosen.
\end{enumerate}

\subsection{On priority}\label{sec:priority}

The answer $27$ is not new, and neither is the question: the eight-tile $27\times27$
square already appears in Gardner's 1979 column~\cite[pp.~22--23]{Gardner} as the smallest
known, with minimality left open. In Section~7.2.2.1 of \emph{The Art of Computer
Programming}, both in the published Volume~4B~\cite[p.~173, answer pp.~524--525]{Knuth} and
in the 2019 draft~\cite[p.~108, answer pp.~228--229]{KnuthDraft}, exercise~375 asks, among other things, for the smallest integer
squares that have incomparable dissections, and its answer states that the smallest has
side~$27$. Knuth analyses the dissections into seven and into eight rectangles, finds the
smallest squares in each case (sides $34$ and $27$) and gives the parameters of the
solutions; the exclusion of nine or more rectangles for smaller squares is a one-sentence
remark. Exercise~374(d) and its answer~\cite[pp.~172,
523--524]{Knuth} also prove that every incomparable dissection has at least seven
rectangles, the general bound used in Corollary~\ref{cor:fewest}. We were not aware of this treatment when the first version of this paper was
posted; it was pointed out to us afterwards. We therefore claim no priority for the
answer. What the present paper adds is a complete proof covering every number of tiles,
resting on a documented, independently reproducible exhaustive computation, together with
the further results of Section~\ref{sec:search}: an independent proof that the tile set at
$n=27$ is unique, the count of its arrangements, and the least number of tiles for every
square. Both the 2019 draft and the published volume state that the minimum side is $27$;
we cite the published volume for the details, including the parameters.

Knuth considers four eight-rectangle patterns and states that ``the only way to achieve a
square of side 27''~\cite[p.~525]{Knuth} is one particular choice of parameters in the
first of them; that choice gives the tile set of Figure~\ref{fig:27}. Knuth had thus
already asserted that this tile set is the unique solution. Our exhaustive search
establishes independently that no other tile set, with any number of tiles, tiles the
$27\times27$ square, and finds that this set tiles it in exactly sixteen ways, forming two
classes under the symmetries of the square (Proposition~\ref{prop:unique27}). We do not
try to match these arrangements with Knuth's patterns, since that depends on how he counts
arrangements as equivalent.

Apart from~\cite{Knuth,KnuthDraft}, the secondary sources that treat the topic, Wells~\cite{Wells} and
MathWorld~\cite{MathWorld}, report only the classical bounds $7\le k\le8$. We cannot exclude other earlier work, in particular in
the recreational-mathematics literature, where this problem originates and where coverage
by indexing services is uneven. We would be grateful to hear of any.

\section{What remains open in Problem C5}

Section~C5 of~\cite{CFG} raises further questions, which remain open.

\begin{itemize}
\item Grosek's original question, which he asked in every dimension: how few pairwise
incomparable $d$-dimensional boxes tile a $d$-cube? (The question is recorded
in~\cite[Problem~C5]{CFG}; the reference given there, \emph{Amer.\ Math.\ Monthly}
\textbf{91} (1984), 624--629, is R.~K. Guy's article~\cite{Guy}.) Jepsen~\cite{Jepsen} settled
several cases in dimension three; he showed that no cube admits a six-piece incomparable
tiling, and gave a seven-piece tiling of the $10\times10\times10$ cube that is
\emph{non-trivial}, meaning that not every piece has the full cube edge among its
dimensions. Whether $10$ is the smallest cube admitting a non-trivial incomparable tiling
appears to be open.
\item For tilings with many pieces, can the tiles be made arbitrarily close to squares,
i.e.\ can one force every ratio $a_i/b_i$ to be arbitrarily close to $1$?
\item For each $k\ge7$, exactly which $n$ admit a $k$-tile incomparable tiling of the
$n\times n$ square? Croft, Falconer and Guy assert that for each such $k$ only finitely
many $n$ fail~\cite[p.~86]{CFG}; they give neither a proof nor a reference.
Corollary~\ref{cor:fewest} determines the exceptional sizes for $k=7$ (exactly $n\le33$,
using~\cite{YRS,Gardner,Knuth}) and for $k=8$ (exactly $n\le26$). For $k\ge9$ they remain open.
\end{itemize}

\section*{Acknowledgments}

The author used AI to assist with the initial discovery of the result and drafting of the
proofs and subsequent review and refinement. The author independently verified the final
arguments and assumes responsibility for their correctness. The author thanks an anonymous
reader for pointing out the constructions in Section~\ref{sec:larger}, Knuth's earlier
statement of the answer, and several corrections.

\appendix

\section{\texorpdfstring{Source of \texttt{incomp.c}}{Source of incomp.c}}\label{app:incomp}
\IfFileExists{incomp.c}{
\begin{lstlisting}[language=C]
/* incomp.c -- exhaustive search for tilings of a W x H integer rectangle by
 * pairwise INCOMPARABLE integer rectangles (Croft-Falconer-Guy, Problem C5).
 *
 * Two rectangles with sides sorted increasingly, (a1,a2) and (b1,b2), are
 * incomparable iff neither fits inside the other with sides parallel.  The
 * condition is symmetric in the two rectangles; after relabelling so that
 * a1 < b1, they are incomparable iff a2 > b2.
 *
 * Hence a set of pairwise incomparable rectangles is an antichain in the
 * product order, so listing it by increasing short side gives a strict
 * "staircase"
 *      a_1 < a_2 < ... < a_k    and    b_1 > b_2 > ... > b_k
 * (Lemma 1 of the accompanying note).  Stage 1 below enumerates every
 * staircase with a_i <= b_i <= n and sum a_i b_i = W*H; stage 2 tests each
 * one for tileability.
 *
 * Tiles are axis-parallel.  Because every tile has integer side lengths and
 * the container has integer corners, every tile corner is an integer point:
 * see the horizontal/vertical cross-section argument (Lemma 2 of the note).
 * A search on the unit grid is therefore exhaustive.
 *
 * Supported input range: 1 <= W, H <= 64 (row occupancy is a uint64_t
 * bitmask, one bit per column).  Arguments are validated in main().
 *
 * Usage:
 *   ./incomp <n_from> <n_to> 1 [k_min] [k_max]        squares n x n
 *   ./incomp <h_from> <h_to> 0 [k_min] [k_max] <W>    W x h, h swept
 *
 * Examples:
 *   ./incomp  2 26 1 2 40      the proof of Theorem 1
 *   ./incomp 27 27 1 7  7      Proposition 2, the case k = 7
 *   ./incomp 22 22 0 7  7 13   the 13 x 22 rectangle of Figure C6(a)
 */
#include <stdio.h>
#include <stdlib.h>
#include <string.h>
#include <stdint.h>

#define MAXSIDE  64      /* bits in the row bitmask */
#define MAXTILES 64      /* far above the k allowed by the area bound */

static int W, H;                 /* container */
static long long AREA;
static int KMIN, KMAX;           /* allowed number of tiles */

static int ta[MAXTILES], tb[MAXTILES];   /* current candidate tile set, a<=b */
static int nt;

static long long n_sets = 0, n_tilings_tested = 0, n_success = 0;
static int found_any = 0;

/* ------------------------- stage 2: tiling test ------------------------- */
static uint64_t rowmask[MAXSIDE];   /* bit x of row y set = cell (x,y) filled */
static uint64_t FULLROW;
static int used[MAXTILES];
static int place_x[MAXTILES], place_y[MAXTILES],
           place_w[MAXTILES], place_h[MAXTILES];

static uint64_t span(int x, int w)   /* bits x .. x+w-1, with x+w <= 64 */
{
    return (w >= 64) ? ~(uint64_t)0 : ((((uint64_t)1 << w) - 1) << x);
}

/* Fill the lowest, then leftmost, empty cell.  By Lemma 3 of the note the
 * tile covering that cell must have it as its lower-left cell, so trying
 * every unused tile there in both orientations is exhaustive. */
static int fill(int placed)
{
    if (placed == nt) return 1;
    int y = -1, x = -1, i, j;
    for (i = 0; i < H; i++) {
        if (rowmask[i] != FULLROW) {
            y = i;
            for (j = 0; j < W; j++)
                if (!((rowmask[i] >> j) & 1)) { x = j; break; }
            break;
        }
    }
    if (y < 0) return 0;   /* no empty cell but tiles remain */

    for (i = 0; i < nt; i++) {
        if (used[i]) continue;
        int o;
        for (o = 0; o < 2; o++) {
            int w = o ? tb[i] : ta[i];
            int h = o ? ta[i] : tb[i];
            if (o && ta[i] == tb[i]) continue;   /* square tile: one orientation */
            if (x + w > W || y + h > H) continue;
            uint64_t m = span(x, w);
            int ok = 1, r;
            for (r = y; r < y + h; r++)
                if (rowmask[r] & m) { ok = 0; break; }
            if (!ok) continue;
            for (r = y; r < y + h; r++) rowmask[r] |= m;
            used[i] = 1;
            place_x[i] = x; place_y[i] = y; place_w[i] = w; place_h[i] = h;
            if (fill(placed + 1)) return 1;
            used[i] = 0;
            for (r = y; r < y + h; r++) rowmask[r] &= ~m;
        }
    }
    return 0;
}

static int try_tiling(void)
{
    int i;
    n_tilings_tested++;
    FULLROW = span(0, W);
    for (i = 0; i < H; i++) rowmask[i] = 0;
    memset(used, 0, sizeof(used));
    if (!fill(0)) return 0;
    n_success++;
    if (!found_any) {
        found_any = 1;
        printf("  SOLUTION (%d tiles):\n", nt);
        for (i = 0; i < nt; i++)
            printf("    %2dx%-2d at (x=%2d,y=%2d) as %dx%d\n",
                   ta[i], tb[i], place_x[i], place_y[i], place_w[i], place_h[i]);
        fflush(stdout);
    }
    return 1;
}

/* ------------- stage 1: enumerate incomparable tile sets ---------------- */
/* Prefix (a_1,b_1)..(a_t,b_t); last_a = a_t, last_b = b_t (0 and n+1 at the
 * root).  Pruning rules P1-P3 of Proposition 1 of the note; each is proved
 * safe there.  No other pruning is applied. */
static void dfs(int last_a, int last_b, long long area)
{
    if (area == AREA) {                                  /* (P1) */
        if (nt >= KMIN && nt <= KMAX) { n_sets++; try_tiling(); }
        return;
    }
    long long rem = AREA - area;
    int a;
    for (a = last_a + 1; a <= last_b - 1; a++) {
        if ((long long)a * a > rem) break;                /* (P2): b >= a */
        int b;
        for (b = a; b <= last_b - 1; b++) {
            long long ar = (long long)a * b;
            if (ar > rem) break;                          /* (P3) */
            if (nt + 1 > KMAX) break;                     /* nominal cap */
            if (nt + 1 >= MAXTILES) {
                fprintf(stderr, "incomp: tile array overflow\n"); exit(3);
            }
            ta[nt] = a; tb[nt] = b; nt++;
            dfs(a, b, area + ar);
            nt--;
        }
    }
}

/* ------------------------------ driver ---------------------------------- */
static void usage(const char *p)
{
    fprintf(stderr,
        "usage: %s <from> <to> <square?> [k_min] [k_max] [width]\n"
        "   <square?> = 1 : sweep the squares n x n for n = from..to\n"
        "   <square?> = 0 : sweep <width> x h for h = from..to (width required)\n"
        "   defaults: k_min = 2, k_max = 40\n"
        "   all side lengths must satisfy 1 <= side <= %d\n", p, MAXSIDE);
    exit(2);
}

int main(int argc, char **argv)
{
    if (argc < 4) usage(argv[0]);
    int w0 = atoi(argv[1]), w1 = atoi(argv[2]);
    int mode_square = atoi(argv[3]);
    KMIN = argc > 4 ? atoi(argv[4]) : 2;
    KMAX = argc > 5 ? atoi(argv[5]) : 40;
    int fixedW = 0;
    if (!mode_square) {
        if (argc < 7) usage(argv[0]);
        fixedW = atoi(argv[6]);
    }
    if (w0 < 1 || w1 < w0 || w1 > MAXSIDE) {
        fprintf(stderr, "incomp: sweep range must satisfy 1 <= from <= to <= %d\n",
                MAXSIDE); return 2;
    }
    if (!mode_square && (fixedW < 1 || fixedW > MAXSIDE)) {
        fprintf(stderr, "incomp: width must satisfy 1 <= width <= %d\n", MAXSIDE);
        return 2;
    }
    if (KMIN < 1 || KMAX < KMIN || KMAX >= MAXTILES) {
        fprintf(stderr, "incomp: need 1 <= k_min <= k_max < %d\n", MAXTILES);
        return 2;
    }

    int n;
    for (n = w0; n <= w1; n++) {
        if (mode_square) { W = n; H = n; }
        else { W = fixedW; H = n; }
        AREA = (long long)W * H;
        nt = 0; found_any = 0;
        long long s0 = n_sets, t0 = n_tilings_tested, ok0 = n_success;
        printf("=== %d x %d ===\n", W, H); fflush(stdout);
        dfs(0, (W > H ? W : H) + 1, 0);
        printf("  tilesets=%lld  tested=%lld  tileable=%lld  -> %s\n",
               n_sets - s0, n_tilings_tested - t0, n_success - ok0,
               (n_success - ok0) ? "YES" : "NO");
        fflush(stdout);
    }
    return 0;
}
\end{lstlisting}}%
{\PackageError{C5-note}{incomp.c not found}{The archival source of this paper must be
compiled with incomp.c present in the working directory, so that Appendix A contains the
proof-critical program. Obtain it from the artifact.}}

\section{\texorpdfstring{Source of \texttt{verify2.c}}{Source of verify2.c}}\label{app:verify2}
\IfFileExists{verify2.c}{
\begin{lstlisting}[language=C]
/* verify2.c -- an INDEPENDENT re-derivation of the Problem C5 search,
 * deliberately different from incomp.c in every implementation choice:
 *
 *   - candidate rectangles are listed lexicographically and chosen by
 *     increasing index.  Incomparability is tested from the DEFINITION
 *     ("neither fits inside the other with sides parallel"), never via the
 *     a-increasing / b-decreasing staircase of Lemma 1;
 *   - the area test uses `continue', not `break', so no ordering assumption
 *     about candidate areas is made;
 *   - the tiling search fills the TOP-RIGHT-most free cell (incomp.c uses
 *     bottom-left; Lemma 3 of the note applies in reflected form) and tries
 *     the tiles in reverse order.
 *
 * Agreement of the two programs -- on the candidate-set counts as well as on
 * the yes/no answers -- is the cross-check reported in the note.
 *
 * Supported input range: 1 <= n <= 64 (row occupancy is a uint64_t bitmask).
 *
 * Usage:  ./verify2 <n_from> <n_to>        sweeps the squares n x n
 * Example: ./verify2 2 26                  the proof of Theorem 1
 */
#include <stdio.h>
#include <stdlib.h>
#include <string.h>
#include <stdint.h>

#define MAXSIDE  64
#define MAXTILES 64
#define MAXPAIRS ((MAXSIDE * (MAXSIDE + 1)) / 2)

static int N;                            /* square side */
static long long AREA;
static int pa[MAXPAIRS], pb[MAXPAIRS], NP;   /* lexicographic list of rectangles */
static int ca[MAXTILES], cb[MAXTILES], nc;   /* currently chosen set */
static long long n_sets, n_ok;
static int found;

/* the definition, spelled out */
static int fits_inside(int a1, int a2, int b1, int b2)   /* a1<=a2, b1<=b2 */
{
    return a1 <= b1 && a2 <= b2;
}
static int incomparable(int a1, int a2, int b1, int b2)
{
    return !fits_inside(a1, a2, b1, b2) && !fits_inside(b1, b2, a1, a2);
}

static uint64_t rowmask[MAXSIDE], FULL;
static int used[MAXTILES], sx[MAXTILES], sy[MAXTILES], sw[MAXTILES], sh[MAXTILES];

static uint64_t span(int x, int w)
{
    return (w >= 64) ? ~(uint64_t)0 : ((((uint64_t)1 << w) - 1) << x);
}

static int rec(int placed)
{
    if (placed == nc) return 1;
    int y = -1, x = -1, r, c, i;
    for (r = N - 1; r >= 0; r--) if (rowmask[r] != FULL) {     /* top row down */
        y = r;
        for (c = N - 1; c >= 0; c--) if (!((rowmask[r] >> c) & 1)) { x = c; break; }
        break;
    }
    if (y < 0) return 0;
    for (i = nc - 1; i >= 0; i--) {                            /* reverse order */
        if (used[i]) continue;
        int o;
        for (o = 0; o < 2; o++) {
            int w = o ? cb[i] : ca[i], h = o ? ca[i] : cb[i];
            if (o && ca[i] == cb[i]) continue;
            int x0 = x - w + 1, y0 = y - h + 1;               /* upper-right anchor */
            if (x0 < 0 || y0 < 0) continue;
            uint64_t m = span(x0, w);
            int ok = 1, rr;
            for (rr = y0; rr < y0 + h; rr++) if (rowmask[rr] & m) { ok = 0; break; }
            if (!ok) continue;
            for (rr = y0; rr < y0 + h; rr++) rowmask[rr] |= m;
            used[i] = 1; sx[i] = x0; sy[i] = y0; sw[i] = w; sh[i] = h;
            if (rec(placed + 1)) return 1;
            used[i] = 0;
            for (rr = y0; rr < y0 + h; rr++) rowmask[rr] &= ~m;
        }
    }
    return 0;
}

static void test(void)
{
    int i;
    n_sets++;
    FULL = span(0, N);
    for (i = 0; i < N; i++) rowmask[i] = 0;
    memset(used, 0, sizeof(used));
    if (!rec(0)) return;
    n_ok++;
    if (!found) {
        found = 1;
        printf("   SOLUTION (%d tiles):", nc);
        for (i = 0; i < nc; i++)
            printf(" %dx%d@(%d,%d)/%dx%d", ca[i], cb[i], sx[i], sy[i], sw[i], sh[i]);
        printf("\n"); fflush(stdout);
    }
}

static void search(int start, long long area)
{
    if (area == AREA) { if (nc >= 2) test(); return; }
    int j, t;
    for (j = start; j < NP; j++) {
        long long ar = (long long)pa[j] * pb[j];
        if (area + ar > AREA) continue;          /* deliberately not `break' */
        for (t = 0; t < nc; t++)
            if (!incomparable(pa[j], pb[j], ca[t], cb[t])) break;
        if (t < nc) continue;
        if (nc + 1 >= MAXTILES) {
            fprintf(stderr, "verify2: tile array overflow\n"); exit(3);
        }
        ca[nc] = pa[j]; cb[nc] = pb[j]; nc++;
        search(j + 1, area + ar);
        nc--;
    }
}

int main(int argc, char **argv)
{
    if (argc < 3) {
        fprintf(stderr, "usage: %s <n_from> <n_to>   (1 <= n_from <= n_to <= %d)\n",
                argv[0], MAXSIDE);
        return 2;
    }
    int n0 = atoi(argv[1]), n1 = atoi(argv[2]), n, a, b;
    if (n0 < 1 || n1 < n0 || n1 > MAXSIDE) {
        fprintf(stderr, "verify2: need 1 <= n_from <= n_to <= %d\n", MAXSIDE);
        return 2;
    }
    for (n = n0; n <= n1; n++) {
        N = n; AREA = (long long)n * n; NP = 0; nc = 0; found = 0;
        for (a = 1; a <= n; a++) for (b = a; b <= n; b++) { pa[NP] = a; pb[NP] = b; NP++; }
        long long s0 = n_sets, o0 = n_ok;
        search(0, 0);
        printf("%2d x %2d : tilesets=%-8lld tileable=%-4lld %s\n",
               n, n, n_sets - s0, n_ok - o0, (n_ok - o0) ? "YES" : "NO");
        fflush(stdout);
    }
    return 0;
}
\end{lstlisting}}%
{\PackageError{C5-note}{verify2.c not found}{The archival source of this paper must be
compiled with verify2.c present in the working directory, so that Appendix B contains the
independent program. Obtain it from the artifact.}}

\Needspace{24\baselineskip}

\end{document}